\documentclass{amsart}
\usepackage[normalem]{ulem}
\usepackage{amssymb,epsfig,mathrsfs,mathpazo,scalerel,url,amsthm,thmtools,bbm}
\usepackage{xcolor}
\usepackage{stmaryrd}
\usepackage{epsfig,bm}
\usepackage[multiple]{footmisc}
\usepackage{accents} 
\usepackage{mathtools} 

\usepackage{numprint}
\usepackage{arydshln, booktabs, makecell, pbox, tabularx}

\usepackage{graphicx, caption, subcaption}
\usepackage[section]{placeins}

\newtheorem{theorem}{Theorem}[section]
\newtheorem{lemma}[theorem]{Lemma}

\newtheorem{corollary}[theorem]{Corollary}

\newtheorem{definition}[theorem]{Definition}

\declaretheorem[style=remark,qed=$\vartriangle$,sibling=theorem]{remark}
\numberwithin{equation}{section}

\newcommand{\eps}{\varepsilon}

\newcommand{\R}{\mathbb R}

\newcommand{\N}{\mathbb N}

\newcommand{\cE}{\mathcal E}

\newcommand{\cL}{\mathcal L}

\newcommand{\Lis}{\cL\mathrm{is}}

\newcommand{\identity}{\mathrm{Id}}

\DeclareMathOperator{\ran}{ran}
\DeclareMathOperator{\supp}{supp}

\DeclareMathOperator*{\argmin}{argmin}

\DeclareMathOperator{\divv}{div}

\newcommand{\nrm}{| \! | \! |}

\newcommand{\new}[1]{#1}

\newcommand{\be}{\begin{equation}}
\newcommand{\ee}{\end{equation}}

\newcommand{\tria}{{\mathcal T}}

\newcommand{\uumlaut}{{\"u}}

{\par\begin{samepage}%
\begin{tabbing}\ttfamily%
 \hspace*{5mm}\=\hspace{3ex}\=\hspace{3ex}\=\hspace{3ex}\=\hspace{3ex}%
\=\hspace{3ex}\=\hspace{3ex}\=\hspace{3ex}\=\hspace{3ex}\kill}%
{\end{tabbing}\end{samepage}}

\makeatletter
\@namedef{subjclassname@2020}{%
  \textup{2020} Mathematics Subject Classification}
\makeatother

\title[A least squares finite element method for backward parabolic problems]{A least squares finite element method for backward parabolic problems}

\date{\today}

\author{Harald Monsuur}

\address{
Korteweg--de Vries (KdV) Institute for Mathematics, University of Amsterdam, P.O. Box 94248, 1090 GE Amsterdam, The Netherlands.
}
\email{harald.monsuur@hotmail.com}

\thanks{This research has been supported by the Netherlands Organization for Scientific Research (NWO)
under contract.~no.~613.009.138. 
We acknowledge the support of SURF (www.surf.nl) in using the National Supercomputer Snellius.}


\begin{document}

\begin{abstract} 
Backward parabolic equations, such as the backward heat equation, are classical examples of ill-posed problems where solutions may not exist or depend continuously on the data. In this work, we study a least squares finite element method to numerically approximate solutions to such problems. We derive conditional stability estimates for the weak formulation of inhomogeneous backward parabolic equations, assuming minimal regularity of the solution. These stability results are then used to establish \emph{a priori} error bounds for our proposed method. We address key computational aspects, including the treatment of dual norms through the construction of suitable test spaces, and iterative solutions. Numerical experiments are used to validate our theoretical findings.
\end{abstract}

\maketitle

\section{Introduction}

This work presents a numerical method for approximating solutions to backward parabolic problems. An example of such a problem is the backward heat equation. Let $J=(0,T)$ be a finite time interval and $\Omega\subset\R^d$ be a bounded Lipschitz domain. We define the space-time cylinder $Q:=J\times\Omega$. Given a source term $f\colon Q\to \R$ and final time data $g\colon \Omega\to \R$, the solution $u\colon Q\to \R$ of the backward heat equation satisfies
\be \label{eq:backward-heat}
\left\{\begin{array}{rl}
\partial_t u-\triangle_x u = f& \text{on }J\times\Omega,\\
u=0& \text{on }J\times \partial \Omega,\\
u(T)=g& \text{on }\Omega.
\end{array}
\right.
\ee

Backward parabolic problems are so-called \emph{ill-posed problems}. There is no continuous dependence of the solution $u$ on the data $(f,g)$, and for some data a solution might not even exist.

\subsection{Conditional stability estimates}
Given that backward parabolic problems are not well-posed, we are interested in obtaining so-called \emph{conditional stability estimates}. These are estimates in which the size of a solution (typically in a weaker norm) depends continuously on its data, provided that the solution satisfies some \emph{a priori} bound.

Conditional stability for backward parabolic problems has been studied extensively. For example, using logarithmic convexity arguments, it was shown in \cite{MR155203} that for the homogeneous heat equation
\begin{align}\label{eq:conditional stability intro}
    \|u(t)\|_{L_2(\Omega)}\leq \|u(0)\|_{L_2(\Omega)}^{\frac{T-t}{T}}\|u(T)\|_{L_2(\Omega)}^{\frac{t}{T}},\quad t\in [0,T].
\end{align}
Other methods for proving conditional stability estimates include energy estimates \cite{MR144083,MR216129} and Carleman estimates \cite{MR3460049}.

More general backward parabolic equations of the form
\begin{align}
    \partial_tu + A(t) u,
\end{align}
admit conditional stability results under certain regularity assumptions on the time-dependent operator  $A(t)$. See, for example, \cite{145, 168.825,MR4473789, MR2520059, MR3348915}. Usually, however, extra smoothness of the solution is required in these statements. 

Conditional stability estimates for general parabolic equations do not exist. In \cite{MR2520059, MR1686671, MR342822}, counterexamples can be found, where $A(t)$ depends only continuously on $t$.

In most of the works mentioned above, conditional stability estimates are provided for homogeneous parabolic equations only. Here, we extend the analysis to inhomogeneous problems, enabling the derivation of \emph{a priori} error estimates for the numerical method described below. 

In this work we will consider the weak formulation of the parabolic equation in the sense of \cite[Ch.~XVIII]{63}. In doing so, we will make minimal smoothness assumptions. For instance, for the heat equation, we will assume the solution $u\in\mathcal{X}:=L_2(J;H_0^1(\Omega))\cap H^1(J;H^{-1}(\Omega))$, when $f\in L_2(J; H^{-1}(\Omega))$ and $g\in L_2(\Omega)$.
In this setting, assuming additional domain regularity, we derive conditional stability estimates in the $L_2(J;H^\beta(\Omega))$-norm, with $\beta\in [0,1)$.
\subsection{Numerical solution}
The numerical solution of backward parabolic problems has been a topic of interest for several decades. Restricting to finite element methods, we mention an incomplete list of some approaches here, in which more references can be found. 

One class of methods comes from the method of quasi-reversibility \cite{185.8}.
Within this class, we mention \cite{19.897}, where through an application of the mixed formulation of quasi-reversibility, the appearance of fourth-order derivatives in the equation are avoided, making their method easily implementable. 
Convergence of the numerical solution in $H^1(Q)$, and for a similar method from \cite{28.5}, in $L_2(J;H^1(\Omega))$, can be shown, but no convergence rate is provided. 

Alternative approaches include methods based on eigenfunction expansions \cite{MR1377249, MR1461788} and the Laplace transform \cite{MR2257113}, which both yield error estimates in the $L_2(\Omega)$-norm for $t\in (0,T]$ for the homogeneous case.

In this work, we propose a \emph{least squares finite element method} defined over the space $\mathcal{X}$. For the heat equation, given some finite-dimensional subspace $\mathcal{X}^\delta\subseteq\mathcal{X}$, we solve
\begin{align}
    u^\delta_\eps:=\argmin_{z^\delta\in \mathcal{X}^\delta}\|\partial_tz^\delta - \triangle_x z^\delta\|_{L_2(J;H^{-1}(\Omega))}^2 + \|z^\delta(T)-g\|_{L_2(\Omega)}^2+ \eps^2\|z^\delta(0)\|_{L_2(\Omega)}^2,
\end{align}
where $\eps>0$ is some regularization parameter. We choose this regularization parameter based on an upper bound for the \emph{data error} and the \emph{approximation error}. The data error is defined as $\|\partial_tu - \triangle_x u - f\|_{L_2(J;H^{-1}(\Omega))} + \|u(t)-g\|_{L_2(\Omega)}$, which can be non-zero since the data may be inexact. The approximation error is defined as $\inf_{z^\delta\in \mathcal{X}^\delta}\|u-z^\delta\|_{\mathcal{X}}$.

Least squares methods of this type have been analyzed in detail in \cite{58.7}, where \emph{a priori} error bounds based on the available conditional stability estimates are proven. Qualitatively, given the available conditional stability estimate, these error bounds are the best possible in terms of the data error and the best-approximation error. 

A key computational challenge in this method is handling dual norms. This is addressed using so-called \emph{test spaces}, whose size depends on the chosen trial space $\mathcal{X}^\delta$. For the backward heat equation with Dirichlet boundary conditions, we prove that a finite element test space defined on the same mesh as $\mathcal{X}^\delta$, with a modest increase in polynomial degree, is sufficient. Numerical experiments suggest that even this increase may not be necessary.

We mention some advantages of our method. The method presented here gives rise to error estimates in the $L_2(\Omega)$-norm, but also in the $L_2(J;H^\beta(\Omega))$-norm, with $\beta\in [0,1)$, if the domain $\Omega$ possesses some additional regularity. Furthermore, given that our method is a finite element method, it can be directly applied to domains with more complex geometries. Thanks to the availability of uniform preconditioners \cite{249.991}, we can obtain solutions very efficiently and we are able to perform large-scale computations \cite{306.65,306.6}.

\subsection{Outline}
In Section~\ref{section:setting} we introduce the general framework for parabolic equations, and state our assumptions. 
In Section~\ref{section:conditional stability estimates} we prove a density result, followed by a regularity estimate that enables conditional stability in higher-order Sobolev norms.
In Section~\ref{sec:least squares} we present the least squares method for solving backward parabolic problems, derive \emph{a priori} error estimates, and describe an iterative solver. We give a concrete example in the form of the heat equation with Dirichlet boundary conditions. 
In Section~\ref{section:numerics}, the behavior of our method is analyzed numerically.

\subsection{Notations}
For normed linear spaces $E$ and $F$, by $\cL(E, F)$ we will denote the normed linear
space of bounded linear mappings $E \rightarrow F$. We write $E \hookrightarrow F$ to denote that $E$ is
continuously embedded into $F$. 

By $C \lesssim D$ we will mean that $C$ can be bounded by a multiple of $D$, unless explicitly stated otherwise, \emph{independently} of parameters which $C$ and $D$ may depend on, such as the discretisation index $\delta$.
Furthermore, $C \gtrsim D$ is defined as $D \lesssim C$, and $C\eqsim D$ as $C\lesssim D$ and $C \gtrsim D$.
\section{Abstract framework and assumptions}\label{section:setting}

In order to analyze conditional stability estimates for a broad class of backward parabolic problems, we begin by introducing an abstract variational framework. This setting generalizes classical examples such as the backward heat equation and allows for more general time-dependent parabolic operators.

\subsection{Abstract framework}

Let $V$ and $H$ be separable Hilbert spaces such that $V \hookrightarrow H$ is dense. Identifying $H$ with its dual $H'$, we obtain the Gelfand triple
\[
V \hookrightarrow H \hookrightarrow V'.
\]
This framework is standard in the theory of evolution equations and provides a natural setting to define weak (variational) solutions. It further allows for the convenient notation $\langle\cdot, \cdot\rangle_H$, both to denote the scalar product on $H\times H$ and its unique continuous
extension to the duality pairing on $V'\times V$. 

We consider a time-dependent sesquilinear form $a(t; \cdot, \cdot) \colon V \times V \to \mathbb{C}$ satisfying:
\begin{itemize}
    \item \textbf{Measurability}: For all $\eta, \zeta \in V$, the map $t \mapsto a(t; \eta, \zeta)$ is measurable on $J = (0,T)$;
    \item \textbf{Boundedness}: There exists $M_a > 0$ such that
    \[
    |a(t; \eta, \zeta)| \leq M_a \|\eta\|_V \|\zeta\|_V;
    \]
    \item \textbf{Coercivity up to shift}: There exist $\alpha > 0$ and $\lambda \in \mathbb{R}$ such that
    \[
    \Re a(t; \eta, \eta) + \lambda \|\eta\|_H^2 \geq \alpha \|\eta\|_V^2.
    \]
\end{itemize}
For almost every $t \in J$, define the linear operator $A(t) \in \mathcal{L}(V, V')$ by
\[
\langle A(t) \eta, \zeta \rangle_H := a(t; \eta, \zeta), \quad \text{for all } \eta, \zeta \in V.
\]
We now define the bilinear form $B \colon \mathcal{X} \to \mathcal{Y}'$ by
\begin{align}\label{eq:definition B}
(Bv)(w) := \int_J  \langle \partial_t v(t), w(t) \rangle_H + \langle A(t) v(t), w(t) \rangle_H  dt,
\end{align}
as well as the trace operator $\gamma_0 u := u(0)$. The function spaces are defined as:
\begin{equation}\label{eq:definition X and Y}
\begin{aligned}
\mathcal{X} &:= L_2(J; V) \cap H^1(J; V'), \\
\mathcal{Y} &:= L_2(J; V).
\end{aligned}
\end{equation}

The space $\mathcal{X}$ is chosen to ensure that both $\partial_t u(t)$ and $A(t) u(t)$ are well-defined in $V'$. Moreover, the embedding $\mathcal{X} \hookrightarrow C(\overline{J}; H)$ ensures that $u(t) \in H$ is well-defined for all $t \in [0, T]$ (see, e.g. \cite[Ch.~XVIII, §1, Th.1]{63}).

It is well known (see, e.g., \cite{63,314.9,247.15}) that the mapping
\begin{align}\label{eq:definition G}
G u := (B u, \gamma_0 u) \colon \mathcal{X} \to \mathcal{Y}' \times H
\end{align}
is an \emph{isomorphism}. This means that for any right-hand side $f \in \mathcal{Y}'$ and initial value $u(0) \in H$, there exists a unique $u \in \mathcal{X}$ satisfying the forward problem.

For convenience, we define the pointwise operator $B(t) \colon \mathcal{X} \to V'$ by
\[
(B(t) v)(w) := \langle \partial_t v(t), w \rangle_H + \langle A(t) v(t), w \rangle_H.
\]
Then, if $B u = f$ in $L_2(J; V')$, we have $B(t) u = f(t)$ for almost every $t \in J$.


\subsection{Additional assumptions}

\new{To establish conditional stability estimates in higher-order space-time Sobolev norms, we need two assumptions.\newline

\paragraph{\textbf{Conditional stability}}

We assume that the following conditional stability estimate holds for the homogeneous problem: when $u \in \mathcal{X}$ satisfies $Bu = 0$, then
\begin{equation}\label{eq:assumption conditional stability}
\|u(t)\|_{H} \lesssim \|u(0)\|_{H}^{1-\omega(t)} \|u(T)\|_{H}^{\omega(t)}, \quad t \in [0, T],
\end{equation}
where $\omega: [0,T] \to [0,1]$ is a \emph{strictly} increasing function with $\omega(0) = 0$ and $\omega(T) = 1$.
\begin{remark}
    Estimate~\eqref{eq:assumption conditional stability} holds with the linear function $\omega(t) = \frac{t}{T}$ in the case where $A(t) \equiv A \in \mathcal{L}(V, V')$ is a constant self-adjoint operator (see Appendix~\ref{sec:appendix} for a proof).

    When $A(t) = g(t) A$, where $g(t)$ is a continuous function satisfying $g(t)\in [g_-, g_+]$, with $g_-,g_+>0$, and $A \in \mathcal{L}(V, V')$ is self-adjoint, the function $\omega$ is given by
    \[
    \omega(t) = \frac{\int_0^t g(s) \, ds}{\int_0^T g(s) \, ds}.
    \]
   This follows from the arguments presented in \cite{MR4379761}.
\end{remark}

\paragraph{\textbf{Lipschitz continuity}}

We will assume that $t\mapsto A(t)\in \mathcal{L}(V,V')$ is Lipschitz continuous, that is,
\begin{align}\label{eq:assumption C1}
    |\left(\left(A(t)-A(s)\right)\eta\right)(\zeta)|\leq L_a |t-s| \|\eta\|_V\|\zeta\|_V,\quad \eta,\zeta\in V, t,s\in J,
\end{align}
for some constant $L_a>0$.
}

\section{Conditional stability estimates for backward parabolic problems}\label{section:conditional stability estimates}

\subsection{Density results}
We aim to establish conditional stability estimates for solutions $u \in \mathcal{X}$ satisfying $Bu = 0$. Later we will use the fact that $G$ defined in \eqref{eq:definition G} is an isomorphism, to establish conditional stability estimates for the inhomogeneous case. 

To establish conditional stability estimates for the homogeneous case, we use an argument that requires additional regularity of $u$. We will prove this additional regularity using a density argument. However, a straightforward approximation of $u$ by elements of $C^\infty(\overline{J};V)$ leads to the difficulty that such approximations generally do not preserve the property $Bu = 0$. 
In this section, we construct explicitly a sequence $\{v_\eps\} \subset C^\infty(\overline{J};V)$ that converges to $u$ in $\mathcal{X}$ and prove that it satisfies additional properties that overcome this difficulty.
\begin{lemma}\label{lemma:dense sequence}
    Let $u\in \mathcal{X}$. Let $\varphi_\eps(t):= t-2\eps \frac{t}{T}+\eps$ and let $\rho\in C^\infty(\R)$ be a smooth function such that $\int_\R \rho=1$ and $\supp \rho = (-1,1)$. Define 
    \begin{equation}\label{eq:definition sequence}
    \begin{aligned}
        v_\eps(t):&=\int_{-1}^1\rho(s) u(\varphi_\eps(t) + \eps s)ds\\
        &=\frac{1}{\eps}\int_J \rho\left(\frac{s-\varphi_\eps(t)}{\eps}\right)u(s)ds,
    \end{aligned}
    \end{equation}
    then for $\eps\to 0$ it holds that
    \begin{itemize}
        \item $v_\eps\to u$ in $\mathcal{X}$,
        \item $\eps \partial_t v_\eps\to 0$ in $L_2(J;V)$,
        \item $\eps \partial_t^2v_\eps\to 0$ in $L_2(J;V')$.
    \end{itemize}
\end{lemma}
\begin{proof}
It is already established in \cite[Thm. 64.36]{70.99} that $v_\eps\to u$ in $\mathcal{X}$. For simplicity we assume that $T=1$. For fixed $\eps$ and $t$, it holds that 
\begin{align*}
    \eps \partial_t v_\eps(t) & = \eps(1-2\eps)\int_{-1}^1\rho(s)\partial_t u(\varphi_\eps(t)+\eps s)ds\\
    & = (1-2\eps)\int_{-1}^1 \rho'(s)u(\varphi_\eps (t)+\eps s)ds\\
    & = -(1-2\eps)\int_{-1}^1 \rho'(s) (u(t)-u(\varphi_\eps(t)+\eps s))ds.
\end{align*}
Hence, we can bound 
\begin{align}\label{eq:bounding argument}
    \|\eps\partial_t v_\eps(t)\|_{V}\lesssim \int_{-1}^1\|u(t)-u(\varphi_\eps(t) + \eps s)\|_{V}ds.
\end{align}
Following the proof from \cite[Theorem~64.36]{70.99}, we conclude that $\eps \partial_t v_\eps\to 0$ in $L_2(J;V)$.
Finally, 
\begin{align}\label{eq:eps dt2 veps}
    \eps \partial^2_t v_\eps(t) = -(1-2\eps)^2\int_{-1}^1\rho'(s)(\partial_t u(t)-\partial_t u(\varphi_\eps(t)+\eps s))ds,
\end{align}
and we can reason as above to conclude that $\eps \partial^2_t v_\eps\to 0$ in $L_2(J;V')$.
\end{proof}
\new{We are now ready to establish an important result for functions $u\in\mathcal{X}$ satisfying $Bu=~0$. Later, in Lemma~\ref{lemma:regularity positive t}, we will use this result to bound the $L_2(J;V')$-norm of $B\partial_t v_\eps$ with $\|u\|_{L_2(J;V)}$, which is a key estimate that leads to the conclusion that $\partial_t u(t)\in H$ for positive values of $t$.}
\begin{lemma}\label{lemma: Bv_eps}
      Let $u\in \mathcal{X}$ satisfy $Bu=0$. Then, for the sequence $v_\eps$ defined in \eqref{eq:definition sequence}, we have 
        \begin{align}\label{eq:Bveps 2}
           (B(t)\partial_t v_\eps)(\zeta)=\langle\frac{-2\eps}{1-2\eps}\partial_t^2v_\eps(t), \zeta(t)\rangle_H + \langle I(t),\zeta(t)\rangle_H,
       \end{align}
       where $I(t)$ satisfies $\|I\|_{L_2(J;V')}\lesssim \|u\|_{L_2(J;V)}$.
\end{lemma}
\begin{proof}
    Thanks to \cite[Cor. 64.14]{70.99} we can interchange the order of duality brackets, applications of operators, and integration. 
    Note that we can write
    \begin{equation}\label{eq:vtt}
    \begin{aligned}
        \partial_t^2 v_\eps(t) &= (1-2\eps)^2\int_{-1}^1 \rho(s)\partial_t^2u(\varphi_\eps(t)+\eps s)ds\\
        &=\frac{(1-2\eps)^2}{\eps}\int_{-1}^1\rho'(s)\partial_t u(\varphi_\eps(t)+\eps s)ds\\
        &=\frac{(1-2\eps)^2}{\eps^2}\int_J\rho'\left(\frac{s-\varphi_\eps(t)}{\eps}\right)\partial_tu(s)ds.
    \end{aligned}
    \end{equation}
    And similarly,
    \begin{equation}\label{eq:A vt}
    \begin{aligned}
        A(t) \partial_t v_\eps(t) &= \frac{(1-2\eps)}{\eps^2}\int_J \rho'\left(\frac{s-\varphi_\eps(t)}{\eps}\right) A(t)u(s)ds\\
        &=\frac{(1-2\eps)}{\eps^2}\int_J\rho'\left(\frac{s-\varphi_\eps(t)}{\eps}\right) (A(t)-A(s))u(s)ds+\\
        &\frac{(1-2\eps)}{\eps^2}\int_J\rho'\left(\frac{s-\varphi_\eps(t)}{\eps}\right) A(s)u(s)ds.
    \end{aligned}
    \end{equation}
    Adding \eqref{eq:vtt} and \eqref{eq:A vt} and using $B(s)u(s)=0$ yields
    \begin{equation}
    \begin{aligned}
        (B(t)\partial_t v_\eps(t))(\zeta) = \langle \frac{-2\eps}{1-2\eps}\partial_t^2v_\eps(t),\zeta(t)\rangle_H + \\\langle \frac{(1-2\eps)}{\eps^2}\int_J\rho'\left(\frac{s-\varphi_\eps(t)}{\eps}\right) (A(t)-A(s))u(s)ds, \zeta(t)\rangle_H.
    \end{aligned}
    \end{equation}
    We will now focus on bounding $I(t):=\frac{1}{\eps^2}\int_J\rho'\left(\frac{s-\varphi_\eps(t)}{\eps}\right) (A(t)-A(s))u(s)ds$ in the $L_2(J;V')$-norm. We have
    \begin{equation}
    \begin{aligned}
        I(t)&=\frac{1}{\eps}\int_{-1}^1\rho'(s) \left(A(t)-A(\varphi_\eps(t)+\eps s)\right)u(\varphi_\eps(t) + \eps s)ds\\
        &=\frac{1}{\eps}\int_{-1}^1\rho'(s) \left(A(t)-A(\varphi_\eps(t)+\eps s)\right)(u(\varphi_\eps(t) + \eps s) - u(t))ds+\\
        &\frac{1}{\eps}\int_{-1}^1\rho'(s) \left(A(t)-A(\varphi_\eps(t)+\eps s)\right)u(t)ds\\
        &=:I_1(t)+I_2(t).
    \end{aligned}
    \end{equation}
    \new{Thanks to \eqref{eq:assumption C1}, we can then bound
    \begin{align}
        \|I_1(t)\|_{V'}\lesssim \frac{1}{\eps}\int_{-1}^1 \eps \|u(\varphi_\eps(t) + \eps s) - u(t)\|_Vds,
    \end{align}
    which is the same upper bound as in \eqref{eq:bounding argument}. This allows us to conclude that $I_1\to 0$ in $L_2(J;V')$.}

    Again, thanks to \eqref{eq:assumption C1}, can bound $\|I_2(t)\|_{V'}$ as
    \begin{equation}
    \begin{aligned}
    \|I_2(t)\|_{V'}\lesssim \frac{1}{\eps}\int_{-1}^1\eps\|u(t)\|_Vds\lesssim \|u(t)\|_V.
    \end{aligned}
    \end{equation}
    We conclude that $\|I\|_{L_2(J;V')}\lesssim \|u\|_{L_2(J;V)}$ for $\eps>0$ small enough.   
\end{proof}
\subsection{Regularity estimate for $t>0$}
The following lemma appears in \cite{145}, but our proof applies to the abstract setting introduced above and assumes only Lipschitz continuity of $A(t)$.
\begin{lemma}[Regularity estimate for $t>0$]\label{lemma:regularity positive t}
    Let $u\in \mathcal{X}$ satisfy $Bu=0$, then for any $t\in (0,T]$ there holds
    \begin{align*}
        \|A(t)u(t)\|_H=\|\partial_tu(t)\|_{H} \lesssim \frac1t\|u(0)\|_{H}.
    \end{align*}
    Furthermore, for any $t_0>0$, $\partial_tu\in C([t_0,T];H)$.
\end{lemma}
\begin{proof}
    Pick any $t_0>0$. Choose $\chi\in C^\infty(\R)$ with $0\leq \chi\leq 1$ and 
    \begin{equation*}
    \chi(t) = \left\{\begin{array}{rl}
0,& t\leq t_0/2,\\
1,& t\geq t_0,
\end{array}
\right.
    \end{equation*}
such that $|\chi'(t)|\lesssim \frac1{t_0}$. Let $v_\eps$ be as in Lemma \ref{lemma:dense sequence} and $I$ be as in Lemma \ref{lemma: Bv_eps}. We have $(B\chi \partial_t v_\eps)(\zeta)=\int_J\langle\chi'\partial_t v_\eps,\zeta\rangle_H + (B\partial_t v_\eps)(\chi\zeta)$ and $\gamma_0 (\chi\partial_tv_\eps)=0$. Hence, for any $t\in[t_0,T]$, using $\mathcal{X}\hookrightarrow C(\overline{J};H)$ and $(B,\gamma_0)\in \Lis(\mathcal{X}, \mathcal{Y}'\times L_2(\Omega))$, we obtain
\begin{align*}
    \|\partial_t v_\eps(t)\|_H&\lesssim \|B\chi \partial_t v_\eps\|_{\mathcal{Y}'}\\
    &\lesssim \|\chi'\partial_t v_\eps\|_{\mathcal{Y}'} + \|B\partial_tv_\eps\|_{\mathcal{Y}'}\\
    &\lesssim \frac1{t_0}\|v_\eps\|_{\mathcal{X}} + \|I\|_{\mathcal{Y}'}+\|\frac{-2\eps}{1-2\eps}\partial_t^2v_\eps\|_{\mathcal{Y}'}\\
    &\lesssim \frac1{t_0}(\|v_\eps(0)\|_H + \|Bv_\eps\|_{\mathcal{Y}'}) + \|u(0)\|_H +\|\frac{-2\eps}{1-2\eps}\partial_t^2v_\eps\|_{\mathcal{Y}'}.
\end{align*}

We now follow the proof of \cite[Ch. XVIII, $\S$1, Th. 1]{63}. Using Lemma~\ref{lemma:dense sequence}, and~\eqref{eq:definition X and Y}, in the limit, the right-hand side is equal to $(\frac{1}{t_0} + 1)\|u(0)\|_{H}$. Therefore, $\partial_t v_\eps$ is a Cauchy sequence in $C([t_0,T];H)$, and since it converges in $\mathcal{Y}'$, its limit must be $\partial_tu$. We conclude that $\partial_t u\in C([t_0,T];H)$ and that $\|\partial_t u(t)\|_{H}\lesssim \frac1{t_0}\|u(0)\|_{H}$ for any $t\in [t_0, T]$.
Since $t_0>0$ was chosen arbitrarily, we obtain the first statement of our lemma.
\end{proof}

\subsection{Case of Sobolev spaces}
In the concrete setting where $H=L_2(\Omega)$ and $V=H_0^1(\Omega)$, we can derive conditional stability estimates in Sobolev spaces with increased smoothness index. However, these estimates depend on elliptic regularity properties, which do not hold in general. The following definition is used to describe the assumptions under which these estimates hold.

\begin{definition}
    Let $\Omega\subset \R^d$ be a bounded Lipschitz domain, and let $A(t)\in \mathcal{L}(V,V')$, $t\in J$. The pair $(\Omega, A)$ is called $(1+\eps)$-regular if $\|u\|_{H^{1+\eps}(\Omega)}\lesssim \|A(t)u\|_{L_2(\Omega)}$ for $u\in H_0^1(\Omega)$.
\end{definition}

\begin{lemma}\label{lemma: regularity space}
    Let $u\in \mathcal{X}$ satisfy $Bu=0$.
    If $H=L_2(\Omega)$ and $V=H_0^1(\Omega)$ and $(\Omega,A)$ is $(1+\eps)$-regular with $\eps>0$, then
    \begin{align*}
        \|u(t)\|_{H^{1+\eps}(\Omega)}\lesssim \frac1t\|u(0)\|_{L_2(\Omega)}.
    \end{align*}
\end{lemma}
\begin{proof}
    Since $Au(t) = \partial_t u(t)\in L_2(\Omega)$, using the $(1+\eps)$-regularity of $\Omega$ and Lemma \ref{lemma:regularity positive t}, we obtain
$$\|u(t)\|_{H^{1+\eps}(\Omega)}\lesssim\|A(t) u(t)\|_{L_2(\Omega)}\lesssim \frac1t\|u(0)\|_{L_2(\Omega)}.$$

\end{proof}
Using interpolation we obtain \new{our new} conditional stability estimates, that are summarized in the next lemma. \new{
\begin{lemma}\label{lemma: higher smoothness cond. stab.}
    Assume that $H=L_2(\Omega)$ and $V=H_0^1(\Omega)$, with $(\Omega, A)$ being $(1+\eps)$-regular. Let $u\in \mathcal{X}$ satisfy $Bu=0$. We have the following estimates:
    \begin{itemize}
        \item For $\beta\in [0,1+\eps]$, we have the bound 
        $$\|u(t)\|_{H^\beta(\Omega)}\lesssim t^{-\frac\beta{1+\eps}}\|u(0)||_{L_2(\Omega)}\left(\frac{\|u(T)\|_{L_2(\Omega)}}{\|u(0)\|_{L_2(\Omega)}}\right)^{(1-\frac{\beta}{1+\eps})\omega(t)},
        $$ 
        for almost every $t\in (0,T]$.
        \item For $\beta\in[0,\frac{1+\eps}{2})$, with $M:=\max\{\|u(0)\|_{L_2(\Omega)},\|u(T)\|_{L_2(\Omega)}+1\}$, we have the bound
        $$
        \|u\|_{L_2(J;H^\beta(\Omega))}^2\lesssim M^2\int_0^T t^{-2\frac\beta{1+\eps}} \left(\frac{\|u(T)\|_{L_2(\Omega)}}{M} \right)^{(2-2\frac{\beta}{1+\eps})\omega(t)}dt.
        $$
        In particular, $\|u\|_{L_2(J;H^\beta(\Omega))}\to 0$, as $\|u(T)\|_{L_2(\Omega)}\to 0$, provided that $\|u(0)\|_{L_2(\Omega)}$ remains bounded.
    \end{itemize}
\end{lemma}
\begin{proof}
    If $(\Omega, A)$ is $(1+\eps)$-regular then we can obtain a conditional stability estimate in the $H^\beta(\Omega)$-norm using an interpolation argument, Lemma \ref{lemma: regularity space}, and by using \eqref{eq:assumption conditional stability}. For $\beta\in [0, \eps)$,
    \begin{equation}
    \begin{aligned}
        \|u(t)\|^2_{H^\beta(\Omega)}&\leq\|u(t)\|^{2\frac{\beta}{1+\eps}}_{H^{1+\eps}(\Omega)}\|u(t)\|_{L_2(\Omega)}^{2-2\frac{\beta}{1+\eps}}\\
        &\lesssim t^{-2\frac\beta{1+\eps}}\|u(0)\|_{L_2(\Omega)}^{2\frac{\beta}{1+\eps}} \|u(0)\|_{L_2(\Omega)}^{(2-2\frac{\beta}{1+\eps})(1-\omega(t))} \|u(T)\|_{L_2(\Omega)}^{(2-2\frac{\beta}{1+\eps})\omega(t)}\\
        &= t^{-2\frac\beta{1+\eps}}\|u(0)||_{L_2(\Omega)}^2\left(\frac{\|u(T)\|_{L_2(\Omega)}}{\|u(0)\|_{L_2(\Omega)}}\right)^{(2-2\frac{\beta}{1+\eps})\omega(t)}.
    \end{aligned}
    \end{equation}

    By integrating this inequality we can find a stability estimate for $\|u\|_{L_2(J;H^\beta(\Omega))}$. For convenience we replace $\|u(0)\|_{L_2(\Omega)}$ by $M:=\max\{\|u(0)\|_{L_2(\Omega)},\|u(T)\|_{L_2(\Omega)}+1\}$, which results in the estimate
    \begin{align}
        \|u\|^2_{L_2(J;H^\beta(\Omega))}
        &\lesssim M^2\int_0^T t^{-2\frac\beta{1+\eps}} \left(\frac{\|u(T)\|_{L_2(\Omega)}}{M} \right)^{(2-2\frac{\beta}{1+\eps})\omega(t)}dt.
    \end{align}
    
    Thanks to the Lebesgue dominated convergence theorem, using $\omega(t)>0$ for $t>0$, and $2\frac{\beta}{1+\eps}<1$, we see that $\|u\|_{L_2(J;H^\beta(\Omega))}\to 0$ as $\|u(T)\|_{L_2(\Omega)}\to 0$, assuming that $\|u(0)\|_{L_2(\Omega)}$ remains bounded.
\end{proof}
\begin{remark}[Estimating the decay rate]\label{remark:decay rate}
    Suppose $\omega(t) \geq C t^\alpha$ on $[0,\delta]$ for some constants $C > 0$, $\alpha > 0$, and $\delta > 0$. Assume that $\|u(0)\|_{L_2(\Omega)}$ is bounded. We can estimate the decay rate of $\|u\|_{L_2(J; H^\beta(\Omega))}$ as $\|u(T)\|_{L_2(\Omega)} \to 0$.
    
    Writing $K = -(2-2\frac{\beta}{1+\beta})C\log\left(\frac{\|u(T)\|_{L_2(\Omega)}}{M}\right) > 0$, we split the upper bound for $\|u\|^2_{L_2(J; H^\beta(\Omega))}$ into two integrals:
    \begin{equation}\label{eq:bound L2Hb splitting}
    \begin{aligned}
        \|u\|^2_{L_2(J; H^\beta(\Omega))} 
        \lesssim M^2 \int_0^\delta t^{-2\frac{\beta}{1+\eps}} e^{-K t^\alpha} \, dt 
        + \\\int_\delta^T t^{-2\frac{\beta}{1+\eps}} \|u(0)\|_{L_2(\Omega)}^{2 - (2-2\frac{\beta}{1+\eps})\omega(t)} \|u(T)\|_{L_2(\Omega)}^{(2-2\frac{\beta}{1+\eps})\omega(t)} \, dt.
    \end{aligned}
    \end{equation}
    
    Since $\|u(0)\|_{L_2(\Omega)}$ is bounded, the second integral can be estimated by a constant multiple of $\delta^{-2\frac{\beta}{1+\eps}} \|u(T)\|_{L_2(\Omega)}^{(2-2\frac{\beta}{1+\eps})\omega(\delta)}$.
    
    To estimate the first integral, set $s = K t^\alpha$. We have
    \begin{align*}
        M^2 \int_0^\delta t^{-2\frac{\beta}{1+\eps}} e^{-K t^\alpha} \, dt 
        &= M^2 \int_0^{K \delta^\alpha} \left( \frac{s}{K} \right)^{-2\frac{\beta}{(1+\eps)\alpha}} e^{-s} \cdot \frac{1}{\alpha} \left( \frac{s}{K} \right)^{1/\alpha - 1} \cdot \frac{1}{K} \, ds \\
        &= M^2 K^{-\frac{1 - 2\frac{\beta}{1+\eps}}{\alpha}} \cdot \frac{1}{\alpha} \int_0^{K \delta^\alpha} s^{\tfrac{1 - 2\frac{\beta}{1+\eps}}{\alpha} - 1} e^{-s} \, ds \\
        &< M^2 K^{-\frac{1 - 2\frac{\beta}{1+\eps}}{\alpha}}\cdot \frac{1}{\alpha} \, \Gamma\left( \frac{1 - 2\frac{\beta}{1+\eps}}{\alpha} \right).
    \end{align*}

    This yields the bound
    \begin{align}
        \|u\|_{L_2(J; H^\beta(\Omega))} 
        \lesssim M K^{-\frac{1 - 2\frac{\beta}{1+\eps}}{2\alpha}} \cdot \sqrt{\frac{1}{\alpha} \,  \Gamma\left( \frac{1 - 2\frac{\beta}{1+\eps}}{\alpha} \right) } 
        + \delta^{-\frac{\beta}{1+\eps}} \|u(T)\|_{L_2(\Omega)}^{(1-\frac{\beta}{1+\eps})\omega(\delta)}.
    \end{align}
    
    We conclude that the decay rate of $\|u\|_{L_2(J; H^\beta(\Omega))}$ as $\|u(T)\|_{L_2(\Omega)} \to 0$, provided $\|u(0)\|_{L_2(\Omega)}$ remains bounded, is governed by
    \begin{align}
        \left( -\log \|u(T)\|_{L_2(\Omega)} \right)^{-\tfrac{1 - 2\frac{\beta}{1+\eps}}{2\alpha}}.
    \end{align}
\end{remark}
}
\subsection{Conditional stability estimates}
We now prove the main theorem of this section. 
\begin{theorem}\label{theorem: conditional stability}
    Let $u\in \mathcal{X}$. For $t\in (0,T)$, it holds that
        \begin{align*}            
        \|u(t)\|_{H}\lesssim(\|u(0)\|_{H}+\|Bu\|_{\mathcal{Y}'})^{1-\omega(t)}(\|u(T)\|_{H} + \|Bu\|_{\mathcal{Y}'})^{\omega(t)}.
        \end{align*}    
    \new{Assume additionally that $H=L_2(\Omega)$ and $V=H_0^1(\Omega)$, with $(\Omega, A)$ being $(1+\eps)$-regular. For $\beta\in [0,\frac{1+\eps}{2})$, with $M:=\min\{\|u(0)\|_{L_2(\Omega)},\|u(T)\|_{L_2(\Omega)}+\|Bu\|_{\mathcal{Y}'} + 1\}$, we have the following conditional stability estimate:
    \begin{align*}
        \|u\|_{L_2(J;H^\beta(\Omega))}\lesssim M\sqrt{\int_0^T t^{-2\frac{\beta}{1+\eps}} \left(\frac{\|u(T)\|_{L_2(\Omega)}+\|Bu\|_{\mathcal{Y}'}}{M} \right)^{(2-2\frac{\beta}{1+\eps})\omega(t)}dt} + \|Bu\|_{\mathcal{Y}'}.
    \end{align*}
    }
\end{theorem}
\begin{proof}
    Let $f=Bu \in \mathcal{Y'}$ and define $\hat u:=G^{-1} (f, 0)$. Then $\tilde u:=u-\hat u$ satisfies $B\tilde u = 0$, so \eqref{eq:assumption conditional stability} applies to $\tilde u$. By the embedding $\mathcal{X}\hookrightarrow C(J;H)$ and the boundedness of $G^{-1}$, we have
    \begin{align}\label{eq:1}
    \|\hat u(t)\|_{H}\lesssim \|\hat u\|_\mathcal{X}\lesssim \|Bu\|_{\mathcal{Y}'},
    \end{align}
    for any $t\in [0,T]$. Furthermore, since $\hat u = 0$, it follows that
    \begin{align}
    \|\tilde u(0)\|_{H}=\|u(0)\|_{H},
    \end{align}
    and
    \begin{align}\label{eq:2}
    \|\tilde u(T)\|_{H}\leq \|u(T)\|_{H}+\|\hat u(T)\|_{H}\lesssim \|u(T)\|_{H}+\|Bu\|_{\mathcal{Y'}}.
    \end{align}
    We can now easily obtain our conditional stability estimates. Using the triangle inequality, \eqref{eq:assumption conditional stability}, \eqref{eq:1} and \eqref{eq:2}, we obtain
        \begin{align*}
            \|u(t)\|_{H}&\lesssim\|\tilde u(t)\|_{H} + \|\hat u(t)\|_{H}\\
            &\lesssim \|\tilde u(0)\|_{H}^{1-\omega(t)}\|\tilde u(T)\|^{\omega(t)} + \|Bu\|_{\mathcal{Y}'}\\
            &\lesssim \|u(0)\|_{H}^{1-\omega(t)}(\|u(T)\|_{H} + \|Bu\|_{\mathcal{Y}'})^{\omega(t)} + \|Bu\|_{\mathcal{Y}'}\\
            &\lesssim (\|u(0)\|_{H}+\|Bu\|_{\mathcal{Y}'})^{1-\omega(t)}(\|u(T)\|_{H} + \|Bu\|_{\mathcal{Y}'})^{\omega(t)}.
        \end{align*}
        The remaining estimate follows analogously.
\end{proof}

\begin{remark}[FOSLS]
    The above conditional stability results also hold when we consider a first-order system least squares formulation of the heat equation.
    Following \cite{75.257, 75.28}, we write $u$ as $u_1$ and introduce ${\bf u} = (u_1, {\bf u}_2)$, where ${\bf u}_2 = -\nabla_x u_1$. We set 
    \begin{align}
        \tilde{\mathcal{X}}:=\{{\bf u} = (u_1, {\bf u}_2)\in L_2(J\times H^1(\Omega))\times L_2(J\times \Omega)^d\colon \divv {\bf u}\in L_2(J\times\Omega)\}.
    \end{align}
    Consider the operator $\tilde B\colon \tilde{\mathcal{X}}\to L_2(J\times \Omega)^d\times L_2(J\times\Omega)$ defined by
    \begin{align}
        \tilde{B} {\bf u} := ({\bf u}_2 + \nabla_x u_1, \divv {\bf u}). 
    \end{align}
    From \cite[Prop. 2.1]{75.28} we see that if $(u_1,{\bf u}_2)\in \tilde{\mathcal{X}}$, then also $u_1\in \mathcal{X}$. Furthermore, following \cite[Section~7.4]{58.7}, we can bound,
    \begin{align*}
        \|Bu_1\|_{\mathcal{Y}'}&=\sup_{0\not=v\in \mathcal{Y}}\frac{\int_J\int_\Omega \partial_tu_1v + \nabla_x u_1\cdot \nabla_xv\,dx\,dt}{\|v||_\mathcal{Y}}
    \end{align*}\begin{align*}
        &=\sup_{0\not=v\in \mathcal{Y}}\frac{\int_J\int_\Omega \divv {\bf u}v+({\bf u}_2+\nabla_x u_1)\cdot\nabla_xv\,dx\,dt}{\|v||_\mathcal{Y}}\\
        &\leq \|\tilde{B}{\bf u}\|_{L_2(J\times \Omega)^{d+1}}.
    \end{align*}
    
    We conclude that Theorem \ref{theorem: conditional stability} also holds for ${\bf u}\in \tilde{\mathcal{X}}$ if we first substitute $\|Bu\|_{\mathcal{Y}'}$ with $\|\tilde{B}{\bf u}\|_{L_2(J\times \Omega)^{d+1}}$ and subsequently, $u$ with $u_1$.
\end{remark}

\section{Least squares approximations for the backward heat equation}\label{sec:least squares}

Our goal is to recover $u(t)$ for $t<T$ using the source term $f$ and end-time data $g$. As mentioned before, for inexact data $(f,g)$, this equation generally has no solution. Therefore, we will only attempt to satisfy \eqref{eq:backward-heat} in a least squares sense. The (regularized) least squares functional we use in this section is motivated by the conditional stability estimates derived earlier. That is, regularization is added to prevent $\|\gamma_0 u\|_{H}$ from blowing up, which allows us to apply the conditional stability results. 

Define $K := (B, \gamma_T)\colon \mathcal{X}\to \mathcal{Y}' \times H$. Let $(\mathcal{X}^\delta)_{\delta\in I}$ be some family of finite-dimensional subspaces of $\mathcal{X}$. For $\delta\in I$ and suitable $0\leq \eps \lesssim 1$, we define the \emph{regularized least squares functional} $G_\eps\colon \mathcal{X}\to \R$ by 
\begin{align}\label{eq:G}
    G_\eps(z):= \|Kz - (f,g)\|^2_{\mathcal{Y}'\times H} + \eps^2\|\gamma_0u\|^2_{H},
\end{align}
and we want to obtain $\bar u_\eps^\delta:=\argmin_{z\in \mathcal{X}^\delta} G_\eps(z)$.

Since we cannot evaluate the norm on $\mathcal{Y}'$, we cannot compute $\bar u_\eps^\delta$. Therefore, we replace this dual norm with a discretized dual norm. Let $(\mathcal{Y}^\delta)_{\delta\in I}$ be a family of finite-dimensional subspaces of $\mathcal{Y}$ that satisfies 
\begin{align}\label{eq:infsup}
    \inf_{\delta \in I}\inf_{\{z\in \mathcal{X}^\delta \colon Bz\not = 0\}}\sup_{0\not= y\in \mathcal{Y}^\delta}\frac{(Bz)(v)}{\|Bz\|_{\mathcal{Y}'}\|v\|_{\mathcal{Y}}}>0.
\end{align}
Upon replacing $\|Bz-f\|_{\mathcal{Y}'}$ in the definition of $G_\eps$ with 
\begin{align}\label{eq:defYdelta}
\|Bz-f\|_{{\mathcal{Y}^\delta}'}:=\sup_{0\not = v\in \mathcal{Y}^\delta}\frac{(f-Bz)(v)}{\|v\|_{\mathcal{Y}}},    
\end{align} 
we obtain a \emph{computable} regularized least squares functional.\footnote{In \eqref{eq:defYdelta}, we will substitute the $\|\cdot\|_\mathcal{Y}$-norm with a different norm $\|\cdot\|_{\mathcal{Y}^\delta}$, that satisfies $\|\cdot\|_\mathcal{Y}\eqsim \|\cdot\|_{\mathcal{Y}^\delta}$ on $\mathcal{Y}^\delta$. Theorem \ref{theorem:residual_estimate}, which is used in the error analysis below, takes this substitution into account. \label{footnote: equivalent norm}}
\begin{align}
    G_\eps^\delta(z):=\|Kz - (f,g)\|_{{\mathcal{Y}^\delta}'\times H}^2 + \eps^2\|\gamma_0u\|^2_{H}.
\end{align}

For $\eps\geq 0$ and $\delta\in I$, the least squares approximation $u_\eps^\delta$ for the backward heat equation is then defined by
\begin{align}\label{eq: definition u}
    u_\eps^\delta := \argmin_{z\in \mathcal{X}^\delta}G_\eps^\delta(z).
\end{align}
Because $K$ is injective, this minimizer exists.

\subsection{Error analysis}\label{sec:error analysis}
The error analysis follows the general framework developed in \cite{58.7}. The following theorem provides an error bound which holds for any $u\in \mathcal{X}$. 
\begin{theorem}\label{theorem:residual_estimate}
    Assume that \eqref{eq:infsup} holds, and let $(f,g)\in \mathcal{Y}' \times H$. Then $u_\eps^\delta$ satisfies, for $u\in \mathcal{X}$,
    \begin{align}\label{eq:res_estimate}
        \|K(u-u_\eps^\delta)\|_{\mathcal{Y}' \times H} + \eps\|\gamma_0(u-u_\eps^\delta)\|_{H}\lesssim \mathcal{E}_{\rm data} + \mathcal{E}_{\rm appr}(\delta) + \eps\|\gamma_0 u\|_{H},
    \end{align}
    where 
    \begin{align*}
        \mathcal{E}_{\rm data} := \|Ku-(f,g)\|_{\mathcal{Y}' \times H}, \quad \mathcal{E}_{\rm appr}(\delta):=\min_{z^\delta\in \mathcal{X}^\delta}\|u-z^\delta\|_\mathcal{X}.
    \end{align*}
\end{theorem}
Choosing $\eps$ appropriately, we obtain the following corollary.
\new{
\begin{corollary}\label{cor:errors}
    If $\eps>0$ is chosen such that
    \begin{align}\label{eq:epsilon_choice}
        \eps\|\gamma_0 u\|_{H}\eqsim \mathcal{E}_{\rm data} + \mathcal{E}_{\rm appr}(\delta),
    \end{align}
    we obtain the following error estimates.\footnote{If we only have $\eps\|\gamma_0 u\|_{H}\gtrsim \mathcal{E}_{\rm data} + \mathcal{E}_{\rm appr}(\delta)$ then the error estimates in Corollary \ref{cor:errors} still hold if we substitute $\mathcal{E}_{\rm data} + \mathcal{E}_{\rm appr}(\delta)$ with $\eps\|\gamma_0 u\|_{H}$.}
    
    For $t\in (0,T)$, it holds that 
    \begin{align}
        \|u(t) - u_\eps^\delta(t)\|_{H}\lesssim (\mathcal{E}_{\rm data} + \mathcal{E}_{\rm appr}(\delta))^{\omega(t)}(\|\gamma_0 u\|_{H} + \mathcal{E}_{\rm data} + \mathcal{E}_{\rm appr}(\delta))^{1-\omega(t)}.
    \end{align}
 
    Assuming that $H=L_2(\Omega)$ and $V=H_0^1(\Omega)$, with $(\Omega, A)$ being $(1+\eps)$-regular. For $\beta\in [0,1)$ with $M:=\min\{\|\gamma_0 u\|_{H}, \mathcal{E}_{\rm data} + \mathcal{E}_{\rm appr}(\delta) + 1\}$, we have
    \begin{equation}
    \begin{aligned}
        \|u - &u_\eps^\delta\|_{L_2(J;H^\beta(\Omega))}\lesssim \\&M\sqrt{\int_0^T t^{-2\frac{\beta}{1+\eps}} \left(\frac{\mathcal{E}_{\rm data} + \mathcal{E}_{\rm appr}(\delta)}{M} \right)^{(2-2\frac{\beta}{1+\eps})\omega(t)}dt} + \mathcal{E}_{\rm data} + \mathcal{E}_{\rm appr}(\delta).
    \end{aligned}
    \end{equation}
\end{corollary}
\begin{proof}
    Substituting \eqref{eq:epsilon_choice} in \eqref{eq:res_estimate} shows that
    $$
    \left\{
    \begin{array}{rcl}
    \|\gamma_0(u-u_\eps^\delta)\|_{H} & \lesssim&  \|\gamma_0 u\|_{H},\\
    \|K(u-u_\eps^\delta)\|_{\mathcal{Y}' \times H} & \lesssim& \cE_{\rm data}+\cE_{\rm appr}(\delta),
    \end{array}
    \right.
    $$
    which we substitute into the conditional stability estimates from Theorem \ref{theorem: conditional stability}.
\end{proof}
\begin{remark}
    Assuming that $\omega(t) \geq C t^\alpha$ on $[0,\delta]$ for some constants $C > 0$, $\alpha > 0$, and $\delta > 0$, from Remark~\ref{remark:decay rate}, we conclude that the decay rate of $\|u - u_\eps^\delta\|_{L_2(J;H^\beta(\Omega))}$ is governed by 
    \begin{align}
        \left( -\log (\cE_{\rm data}+\cE_{\rm appr}(\delta)) \right)^{-\tfrac{1 - 2\frac{\beta}{1+\eps}}{2\alpha}}.
    \end{align}
\end{remark}
}
\subsection{Iterative solution}\label{sec:iterative solution}

In this section, we discuss how to obtain $u_\eps^\delta$ using the Preconditioned Conjugate Gradient (PCG) method as an approximation to some $u\in \mathcal{X}$. We assume that our regularization parameter satisfies $\eps \|\gamma_0 u\|_{H}\eqsim \mathcal{E}_{\rm data} + \mathcal{E}_{\rm appr}(\delta)$.
\footnote{Choosing $\eps$ appropriately, would in practice require some knowledge of the size of the data error and the norm of $\gamma_0 u$.}

We further assume that there exist (uniform) preconditioners $G_\mathcal{Y}^\delta = {G_\mathcal{Y}^\delta}'\in \Lis ({\mathcal{Y}^\delta}',\mathcal{Y}^\delta)$ and $G_\mathcal{X}^\delta = {G_\mathcal{X}^\delta}'\in \Lis ({\mathcal{X}^\delta}',\mathcal{X}^\delta)$ satisfying 
\begin{align}\label{eq:preconditioners}
    \|G_\mathcal{W}^\delta f\|_\mathcal{W}\eqsim f(G_\mathcal{W}^\delta f)\quad(f\in {W^\delta}', \delta\in I), 
\end{align}
for $\mathcal{W}\in \{\mathcal{X}, \mathcal{Y}\}$. 

Such a preconditioner can be used to \emph{define} the inner product
\begin{align}\label{eq:new inner-product}
    \langle v, \tilde v\rangle_{\mathcal{W}^\delta}:= ((G_\mathcal{W}^\delta)^{-1}\tilde v)(v)\quad (v,\tilde v\in \mathcal{W}^\delta),
\end{align}
which, thanks to \eqref{eq:preconditioners}, satisfies $\|\cdot \|_{\mathcal{W}^\delta}\eqsim \|\cdot \|_\mathcal{W}$ on $\mathcal{W}^\delta$.

Further, since by definition the $G_\mathcal{W}^\delta$ is the Riesz lifter associated to the Hilbert space $(\mathcal{W}^\delta, \langle\cdot, \cdot\rangle_{\mathcal{W}^\delta})$, it follows that for $f\in {\mathcal{W}^\delta}'$,
\begin{align}\label{eq:discrete dual norm}
    \|f\|^2_{{\mathcal{W}^\delta}' }:=\sup_{0\not =v\in\mathcal{W}^\delta}\frac{|f(v)|^2}{\|v\|_{\mathcal{W}^\delta}^2}=\|G_\mathcal{W}^\delta f\|^2_{\mathcal{W}^\delta}=f(G_\mathcal{W}^\delta f).
\end{align}

The preconditioner $G_\mathcal{Y}^\delta$ is used to replace the norm $\|\cdot\|_\mathcal{Y}$ on $\mathcal{Y}^\delta$ with $\|\cdot \|_{\mathcal{Y}^\delta}$, allowing us to reformulate \eqref{eq: definition u} as a symmetric positive definite (SPD) system. The preconditioner $G_\mathcal{X}^\delta$ is then used as a preconditioner for this system.

\subsubsection{Obtaining an SPD system}
Following footnote\footref{footnote: equivalent norm}, we replace $\|\cdot \|_\mathcal{Y}$ with $\|\cdot\|_{\mathcal{Y}^\delta}$ in the definition of $u^\delta_\eps$, without affecting the quality of our solution $u_\eps^\delta$. 

Thanks to \eqref{eq:discrete dual norm}, we can equivalently compute $u_\eps^\delta$ as 
\begin{align*}
    u_\eps^\delta = \argmin_{z\in \mathcal{X}^\delta}\{(Bz - f)(G_\mathcal{Y}^\delta (Bz - f)) + \|\gamma_Tz - g\|^2 + \eps^2 \|\gamma_0 z\|^2\},
\end{align*}
which, by employing the Euler-Lagrange equations, is equivalent to 
\begin{equation}\label{eq:spd system}
\begin{gathered}
(S^\delta_\eps u_\eps^\delta)(\tilde z^\delta):=(Bu_\eps^\delta)(G_\mathcal{Y}^\delta B\tilde z^\delta) + \langle \gamma_T \tilde z^\delta, \gamma_T u_\eps^\delta\rangle_{H} + \eps^2\langle \gamma_0 \tilde z^\delta, \gamma_0 u_\eps^\delta\rangle_{H} \\= f(G_\mathcal{Y}^\delta B\tilde z^\delta) + \langle \gamma_T \tilde z^\delta, g\rangle_{H}=:h^\delta(\tilde z^\delta)\quad (\tilde z^\delta \in \mathcal{X}^\delta).
\end{gathered}
\end{equation}

\subsubsection{Preconditioning the SPD system}
We now show that the operator $G_\mathcal{X}^\delta$ can be used as a preconditioner for the SPD system \eqref{eq:spd system}.  
For convenience we define an $\eps$-dependent norm on $\mathcal{X}$,
\begin{align}
\nrm \cdot \nrm^2_\eps:=\|K\cdot\|^2_{\mathcal{Y}' \times H} + \eps \|\gamma_0\cdot\|^2_{H},
\end{align}
and for $z^\delta\in \mathcal{X}^\delta$ we set
\begin{align}
\nrm z^\delta \nrm^2_{\eps, \delta}:=(Bz^\delta)(G_\mathcal{Y}^\delta B z^\delta) + \langle \gamma_Tz^\delta, \gamma_Tz^\delta\rangle_{H} + \eps^2\langle \gamma_0 z^\delta, z^\delta\rangle_{H}.
\end{align}
Then thanks to \eqref{eq:infsup}, for $z^\delta \in \mathcal{X}^\delta$, it holds that $\nrm z^\delta\nrm_\eps\eqsim \nrm z^\delta\nrm_{\eps, \delta}$. Furthermore, thanks to $(B, \gamma_0)\in \Lis(\mathcal{X}, \mathcal{Y}\times H)$ and $\|\cdot\|_{\mathcal{X}^\delta}\eqsim \|\cdot\|_\mathcal{X}$ on $\mathcal{X}^\delta$, we obtain the following equivalence of norms on $\mathcal{X}^\delta$
\begin{align}\label{eq:equivalence norms}
    \eps^2 \|\cdot \|_{\mathcal{X}^\delta}\lesssim \nrm \cdot \nrm^2_{\eps,\delta}\lesssim \max(1, \eps^2) \|\cdot\|_{\mathcal{X}^\delta}.
\end{align}
This means that $G_\mathcal{X}^\delta$ can be used as a preconditioner for \eqref{eq:spd system}, although with an unfavorable $\eps$-dependence on the condition number.
\subsubsection{Stopping criterion}
Let $\tilde u_\eps^\delta\in \mathcal{X}^\delta$ be an intermediate solution produced by the PCG method. To obtain a stopping criterion, we need some condition to determine when the quality of the intermediate solution is satisfactory. The error analysis in Section ~\ref{sec:error analysis} hinges on the validity of \eqref{eq:res_estimate}. Hence, if $\tilde u_\eps^\delta$ also satisfies this bound, then the error analysis for this function is qualitatively the same as the error analysis for $u^\delta_\eps$.

We can bound the terms appearing in \eqref{eq:res_estimate} as follows
\begin{align*}
    \|K(u - \tilde u^\delta_\eps)\|_{\mathcal{Y}' \times H} + \eps\|\gamma_0(u-\tilde u^\delta_\eps)\|_{H} &\lesssim \nrm u - u_\eps^\delta\nrm_\eps + \nrm u_\eps^\delta - \tilde u_\eps^\delta\nrm_{\eps, \delta}\\
    &\lesssim \mathcal{E}_{\rm data} + \mathcal{E}_{\rm appr}(\delta) + \nrm u_\eps^\delta - \tilde u_\eps^\delta\nrm_{\eps, \delta},
\end{align*}
and we see that we want to achieve 
\begin{align}\label{eq:stopping criterion 1}
    \nrm u_\eps^\delta - \tilde u_\eps^\delta\nrm_{\eps, \delta} \lesssim \mathcal{E}_{\rm data} + \mathcal{E}_{\rm appr}(\delta).
\end{align}
To obtain our stopping criterion we bound $\nrm u_\eps^\delta - \tilde u_\eps^\delta\nrm_{\eps, \delta}$ as follows
\begin{align*}
    \nrm u_\eps^\delta - \tilde u_\eps^\delta\nrm_{\eps, \delta} &= \sup_{0\not = v^\delta\in \mathcal{X}^\delta}\frac{(h^\delta - S_\eps^\delta \tilde u_\eps^\delta)(v^\delta)}{\nrm v^\delta\nrm_{\eps,\delta}}\\
    &\lesssim \eps^{-1}\sup_{0\not = v^\delta\in \mathcal{X}^\delta}\frac{(h^\delta - S_\eps^\delta \tilde u_\eps^\delta)(v^\delta)}{\|v^\delta\|_{\mathcal{X}^\delta}}\\
    &=\eps^{-1}(h^\delta-S_\eps^\delta \tilde u_\eps^\delta)(G_\mathcal{X}^\delta(h^\delta - S_\eps^\delta \tilde u_\eps^\delta)),
\end{align*}
where we used \eqref{eq:equivalence norms} and \eqref{eq:discrete dual norm}. We see that \eqref{eq:stopping criterion 1} is achieved if
\begin{align*}
    (h^\delta-S_\eps^\delta \tilde u_\eps^\delta)(G_\mathcal{X}^\delta(h^\delta - S_\eps^\delta \tilde u_\eps^\delta)) \lesssim \eps (\mathcal{E}_{\rm data} + \mathcal{E}_{\rm appr}(\delta))\eqsim \frac{(\mathcal{E}_{\rm data} + \mathcal{E}_{\rm appr}(\delta))^2}{\|\gamma_0 u\|_{H}},
\end{align*}
which is the stopping criterion we will use in our numerical experiments.
\begin{remark}
    The first inequality in \eqref{eq:equivalence norms} is expected to be very crude on finite element spaces. Other methods could also be used to estimate $\nrm u_\eps^\delta - \tilde u_\eps^\delta\nrm_{\eps, \delta}$ using data obtained from the PCG method (see \cite{75.56, 203.6, 15.91}). 
\end{remark}

\subsection{Concrete setting}\label{sec:concrete setting}
To finish this section, we show how to apply the least squares method to the backward heat equation with Dirichlet boundary conditions. The backward heat equation with Dirichlet boundary conditions reads as, find $u\colon J\times\Omega\to \R$ that satisfies 
\be \label{eq:backward-heat2}
\left\{\begin{array}{rl}
\partial_t u-\triangle_x u = f& \text{on }J\times\Omega,\\
u=0& \text{on }J\times \partial \Omega,\\
u(T)=g& \text{on }\Omega,
\end{array}
\right.
\ee
where $\Omega\subset \R^d$ is a Lipschitz domain. In the associated functional framework, we set $H=L_2(\Omega)$, $V=H_0^1(\Omega)$. The bilinear form is defined by $a(t; \eta, \zeta):=\int_\Omega \nabla_x \eta\cdot \nabla_x\zeta$. 

We need to consider the choice of the family of \emph{trial spaces} $(\mathcal{X}^\delta)_{\delta\in I}$ and \emph{test spaces} $(\mathcal{Y}^\delta)_{\delta\in I}$. In order for these families to be valid choices, we need inf-sup stability ~\eqref{eq:infsup} to be satisfied. Otherwise, the error analysis from Section ~\ref{sec:error analysis} may not hold. Furthermore, an efficient iterative method requires the existence of preconditioners satisfying ~\eqref{eq:preconditioners}.

Let $(\tria^\delta)_{\delta\in\Delta}$ be a family of conforming, uniformly shape regular partitions of $\overline{\Omega}$ into $d$-simplices. Let $(J^\delta)_{\delta\in\Delta}$ be a partition of $\overline{J}$. We define 
\begin{align*}
\mathcal{S}_{\tria^\delta}^{0,p}:=\{z\in C(\overline{\Omega})\colon z_{|T}\in \mathcal{P}_p(T); (T\in \tria^\delta)\},    
\end{align*}
and
\begin{align*}
    \mathcal{S}_{J^\delta}^{-1,p}:=\{z\in L_2(J)\colon z_{|e}\in P_p(e);(e\in J^\delta)\},
\end{align*}
and $\mathcal{S}_{J^\delta}^{0,p}:=C(\overline{J})\cap S_{J^\delta}^{-1,p}$.

For $p\in \N$, we then define our \emph{trial} and \emph{test} spaces as follows,
\begin{equation}\label{eq:definition fem spaces}
\begin{aligned}
    \mathcal{X}^\delta &:= \mathcal{X}\cap (S_{J^\delta}^{0,p}\otimes \mathcal{S}_{\tria^\delta}^{0,p}),\\
    \mathcal{Y}^\delta &:=\mathcal{S}_{J^\delta}^{-1,p}\otimes (\mathcal{S}_{\tria^\delta}^{0,p+d+1}\cap H_0^1(\Omega)).
\end{aligned}
\end{equation}
\subsubsection{Verification of \eqref{eq:infsup}}
The next lemma shows that we have inf-sup stability for the above choice of trial and test spaces.
\begin{lemma}\label{lemma:infsup}
    For any $p\in \N$, \eqref{eq:infsup} holds.
\end{lemma}
\begin{proof}
    According to Theorem 4.3 in \cite{58.7}, proving this lemma amounts to the construction of a family of Fortin interpolators $Q^\delta\colon \mathcal{Y}\to \mathcal{Y}^\delta$, with $\sup_{\delta\in\Delta}\|Q^\delta\|_{\mathcal{L}(\mathcal{Y},\mathcal{Y})}<\infty$, such that 
    \begin{align}\label{eq:Fortin}
        (B\mathcal{X}^\delta)((\identity-Q^\delta)\mathcal{Y}) = 0.
    \end{align} 
     Our Fortin interpolator will be of the form $Q^\delta=Q^\delta_t\otimes Q^\delta_x$.
     
    Since $L_2(J)\otimes H_0^1(\Omega)$ is isomorphic to $\mathcal{Y}$, we only need to verify \eqref{eq:Fortin} for functions of the form $v = v_t\cdot v_x$ with $v_t\in L_2(\Omega)$ and $v_x\in H_0^1(\Omega)$. 
    Let $z^\delta = z_t^\delta\cdot z_x^\delta\in \mathcal{X}^\delta$, we then compute
    \begin{align*}
        (Bz^\delta)((\identity-Q^\delta) v)) &= (Bz^\delta)((\identity-Q_t^\delta)\otimes\identity \;v)+(Bz^\delta)(Q_t^\delta\otimes(\identity-Q_x^\delta)v).
    \end{align*}
    
    The first term is equal to
    \begin{align*}
        \sum_{e\in J^\delta}\sum_{T\in T^\delta}\left\{\int_e\partial_tz^\delta_t(\identity-Q_t^\delta)v_tdt\int_Tz_x^\delta v_xdx + \int_e z_t^\delta(\identity-Q_t^\delta )v_tdt\int_T\nabla_xz_x^\delta \nabla_xv_xdx\right\},
    \end{align*}
    which clearly vanishes for $Q_t^\delta$ being the $L_2$-orthogonal projector onto $S^{-1,p}_{J^\delta}$.

    The second term is equal to
    \begin{align*}
        \sum_{e\in J^\delta}\sum_{T\in T^\delta}&\biggl\{
        \int_e
        \partial_tz^\delta_tQ_t^\delta v_tdt
        \int_T z_x^\delta (\identity-Q_x^\delta)v_x dx + \\
        &\int_e z_t^\delta Q_t^\delta v_t dt \biggl(\int_T-\triangle_x z_x^\delta (\identity - Q_x^\delta )v_x dx+\int_{\partial_T}\frac{\partial z_x^\delta}{\partial n}(\identity - Q_x^\delta )v_x ds\biggr)
       \biggr\}.
    \end{align*}
    
   From Section 4.1 from \cite{204.19}, it is known that there exists a Fortin interpolation operator $Q_{x}^\delta\colon H_0^1(\Omega)\to \mathcal{S}_{\tria^\delta}^{0,p + d + 1}\cap H_0^1(\Omega)$, with $\sup_{\delta\in \Delta}\|Q_x^\delta\|_{\Lis(H_0^1(\Omega),H_0^1(\Omega))}<\infty$, such that
   $$
   \ran (\identity - Q_{x}^\delta )|_e\perp_{L_2(e)}\mathcal{P}_{p+1}(e)\;\;\; (e\in \mathcal{F}(\tria^\delta)),
   $$
   and 
   $$
    \ran (\identity - Q_{x}^\delta )|_K\perp_{L_2(K)}\mathcal{P}_{p}(K)\;\;\; (K\in \tria^\delta).
   $$
\end{proof}
\subsubsection{Existence of preconditioners}\label{sec:preconditioners}
Obtaining the preconditioner $G_\mathcal{Y}^\delta$ is straight-forward, see also \cite{58.6}. We equip $\mathcal{S}_{J^\delta}^{-1,p}$ with a $L_2(J)$-orthonormal basis $\Phi_t$ and define $G_{\mathcal{Y},t}^\delta \colon (\mathcal{S}_{J^\delta}^{-1,p})' \to \mathcal{S}_{J^\delta}^{-1,p}$ by $G_{\mathcal{Y},t}^\delta f = \sum_{\phi_t\in \Phi_t}f(\phi_t)\phi_t$. We define $G_{\mathcal{Y},x}^\delta\colon (\mathcal{S}_{\tria^\delta}^{0,p+d+1}\cap H_0^1(\Omega))' \to \mathcal{S}_{\tria^\delta}^{0,p+d+1}\cap H_0^1(\Omega)$ as a symmetric spatial multigrid solver that satisfies \eqref{eq:preconditioners} with $\mathcal{W}$ reading as $H_0^1(\Omega)$. Then it follows that $G_\mathcal{Y}^\delta:=G_{\mathcal{Y},t}^\delta\otimes G_{\mathcal{Y},x}^\delta$ satisfies \eqref{eq:preconditioners}, with $\mathcal{W}=\mathcal{Y}$.

A preconditioner $G_\mathcal{X}^\delta$ that is constructed using a symmetric spatial multigrid solver and a wavelet basis in time that is stable in $L_2(J)$ and $H^1(J)$ can be found in \cite{249.991}.
\section{Numerical experiments}\label{section:numerics}

In this section, we investigate our regularized least squares formulation for solving the backward heat equation. Let $Q$ denote space-time domain $J\times\Omega \subset \R^{d+1}$, where $\Omega=(0,1)^d\subset \R^d$ is the spatial domain and $J=(0,T)$ is the temporal interval. We will specify $T$ in the examples below. For data $f$ and $g$ we consider the backward heat problem of finding $u\in \mathcal{X}$ that satisfies in distributional sense
\begin{equation*}
\left\{\begin{array}{rl}
\partial_t u-\triangle_x u = f& \text{on }J\times\Omega,\\
u=0& \text{on }J\times \partial \Omega,\\
u(T)=g& \text{on }\Omega.
\end{array}
\right.
\end{equation*}

In case $d=2$, let $\mathbbm T$ denote the set of all conforming triangulations generated via \emph{newest vertex bisection}, starting from an initial triangulation $\tria_0$ consisting of four triangles formed by cutting $(0,1)^2$ along its diagonals. The interior vertex in this initial mesh is designated as the newest vertex for all four triangles. 

In the case $d = 3$, let $\mathbbm{T}$ denote the set of all conforming tetrahedral meshes generated via \emph{newest vertex bisection}, starting from an initial triangulation $\tria_0$ of $(0,1)^3$. 
The initial mesh $\tria_0$ is constructed by first decomposing $\Omega$ into $6$ tetrahedra using the Kuhn splitting and subsequently performing one uniform bisection to obtain an interior vertex.

Additionally, let $\mathbbm J$ denote the collection of all one-dimensional temporal grids that are locally quasi-uniform.

Following Section~\ref{sec:concrete setting}, for any $\tria^\delta\in\mathbbm T$ and $J^\delta\in \mathbbm J$, we define the discrete \emph{trial} and \emph{test} spaces as 
\begin{equation}\label{eq:def l}
\begin{aligned}
\mathcal{X}^\delta &:= \mathcal{X} \cap \left( S_{J^\delta}^{0,1} \otimes \mathcal{S}_{\tria^\delta}^{0,1} \right),\\\mathcal{Y}^\delta &:= \mathcal{S}_{J^\delta}^{-1,1} \otimes \left( \mathcal{S}_{\tria^\delta}^{0,1+l} \cap H_0^1(\Omega) \right), 
\end{aligned}
\end{equation} 
where we initially choose $l=d+1$. Then the numerical approximation $u_\eps^\delta\in\mathcal{X}^\delta$ of $u$ satisfies
\begin{equation*}
\begin{gathered}
(Bu_\eps^\delta)(G_\mathcal{Y}^\delta B\tilde z^\delta) + \langle \gamma_T \tilde z^\delta, \gamma_T u_\eps^\delta\rangle_{L_2(\Omega)} + \eps^2\langle \gamma_0 \tilde z^\delta, \gamma_0 u_\eps^\delta\rangle_{L_2(\Omega)} \\= f(G_\mathcal{Y}^\delta B\tilde z^\delta) + \langle \gamma_T \tilde z^\delta, g\rangle_{L_2(\Omega)}\quad (\tilde z^\delta \in \mathcal{X}^\delta).
\end{gathered}
\end{equation*}
We obtain the solution to this system using the PCG method, where we use $G_\mathcal{X}^\delta$ as a preconditioner. The preconditioners $G_\mathcal{X}^\delta$ and $G_\mathcal{Y}^\delta$ are chosen as in Section ~\ref{sec:preconditioners}

Throughout this section, the meshes $J^\delta$ and $\tria^\delta$ are generated by dividing the time interval $J$ into $2^k$ equal subintervals and applying $d\cdot k$ uniform bisections to the initial spatial triangulation $\tria_0$.

In our experiments we prescribe the solution $u(t,x,y):=(1 + t^3)\sin(\pi x)\sin(\pi y)$ in the case $d=2$, and $u(t,x,y,z):=(1+t^3)\sin(\pi x)\sin(\pi y)\sin(\pi z)$ in the case $d=3$, and compute the (exact) data $f$ and $g$ correspondingly.

The numerical experiments were carried out using code adapted from \cite{306.65}.
\subsection{Choice of $l$}
In Lemma ~\ref{lemma:infsup} we showed that for $l=d+1$ the inf-sup condition \eqref{eq:infsup} is satisfied. However, this result is not expected to be sharp. In fact, it turns out that we may choose $l=0$.

Let $\mathcal{Y}^\delta$ and $\hat{\mathcal{Y}}^\delta$ be chosen as in \eqref{eq:def l}, with $l=0$ and $l=d+1$, respectively, and denote their respective preconditioners by $G_\mathcal{Y}^\delta$ and $G_{\hat{\mathcal{Y}}}^\delta$.
Let $\{\varphi_1,\varphi_2,\ldots\}$, $\{\psi_1,\psi_1,\ldots\}$ and $\{\hat\psi_1,\hat\psi_2,\ldots\}$ be bases for $\mathcal{X}^\delta$, $\mathcal{Y}^\delta$ and $\hat{\mathcal{Y}}^\delta$, respectively. Define the matrices ${\bf B}, {\bf \hat B}, {\bf G}$ and ${\bf \hat G}$ by ${\bf B}_{ij} = (B \varphi_j)(\psi_i)$, ${\bf \hat B}_{ij} = (B \varphi_j)(\hat\psi_i)$, $({\bf G}^{-1})_{ij} = ((G_\mathcal{Y}^\delta)^{-1} \psi_j)(\psi_i)$, and $({\bf \hat G}^{-1})_{ij} = ((G_{\hat{\mathcal{Y}}}^\delta)^{-1} \psi_j)(\psi_i)$.
Noticing that $Bz\not =0$ for all $z\in \mathcal{X}^\delta\setminus\{0\}$, $\|Bz\|_{\mathcal{Y}'}\eqsim \|Bz\|_{{\hat{\mathcal{Y}}^\delta}'}$ for $z\in \mathcal{X}^\delta$, and using \eqref{eq:preconditioners}, we deduce that
\begin{align*}
    \gamma^\delta:=&\inf_{\{z\in \mathcal{X}^\delta \colon Bz\not = 0\}}\sup_{0\not= y\in \mathcal{Y}^\delta}\frac{(Bz)(v)}{\|Bz\|_{\mathcal{Y}'}\|v\|_{\mathcal{Y}}} \\=& 
    \inf_{0\not=z\in\mathcal{X}^\delta}\sup_{0\not= y\in \mathcal{Y}^\delta}\frac{(Bz)(v)}{\|Bz\|_{\mathcal{Y}'}\|v\|_{\mathcal{Y}}}\\
    \eqsim&\inf_{0\not=z\in\mathcal{X}^\delta}\sup_{0\not= y\in \mathcal{Y}^\delta}\frac{(Bz)(v)}{[(Bz)(G_{\hat{\mathcal{Y}}}^\delta Bz)]^{\frac12} [((G_\mathcal{Y}^\delta)^{-1} v)(v)]^\frac12}.
\end{align*}
It follows that $(\gamma^\delta)^2$ is equivalent to the smallest generalized eigenvalue of the pencil $({\bf B}^T{\bf G}{\bf B}, {\bf \hat B}^T {\bf \hat G} {\bf \hat B})$ (i.e., ${\bf B}^T{\bf G}{\bf B}{\bf x}=\lambda {\bf \hat B}^T {\bf \hat G} {\bf \hat B} {\bf x}$, for some vector ${\bf x}\not=0$), which is computable.

In the case $d=2$ with $T=1$, for the sequence of meshes described above, we verified that the smallest generalized eigenvalue of this pencil remains about constant. This implies that we may use $l=0$ in the definition of $\mathcal{Y}^\delta$, which is done in the remainder of this section. 

\subsection{Exact data}
We now solve the backward heat equation with unperturbed data. We set the temporal domain $J=(0,1)$, i.e. $T=1$. Since the prescribed solution is smooth, we opt for the regularization parameter $\eps = {\rm DoFs}^{-\frac1d}$ for which we have $\mathcal{E}_{\rm appr}(\delta)\lesssim \eps$. 
\subsubsection{The case $d=2$}
We first consider the case $d=2$. In the left picture of Figure \ref{fig:example 1} the $L_2(\Omega)$-errors of the numerical approximation of $u(t)$ for different values of $t$ are shown. For $t$ close to $T=1$, the convergence rate is about $\mathcal{O}({\rm DoFs}^{-\frac23})$, while for smaller $t$ this convergence gradually weakens.

In the right picture of Figure \ref{fig:example 1} the errors in the space-time $L_2(J; H^1(\Omega))$-norm are given. Surprisingly, the error in this norm converges at an optimal rate $\mathcal{O}({\rm DoFs}^{-\frac13})$.

\begin{figure}[h!]
    \begin{subfigure}{.5\textwidth}
    \centering
    \includegraphics[width=\linewidth]{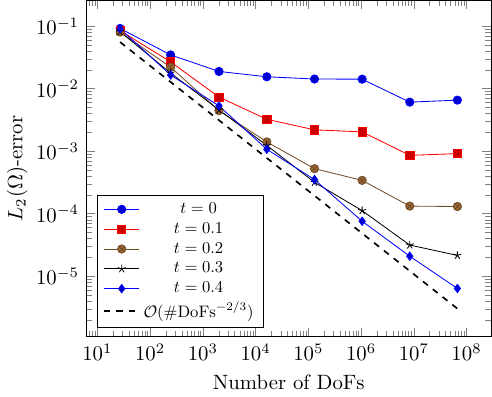}
    \end{subfigure}\hspace*{0cm}%
    \begin{subfigure}{.5\textwidth}
    \centering
    \includegraphics[width=\linewidth]{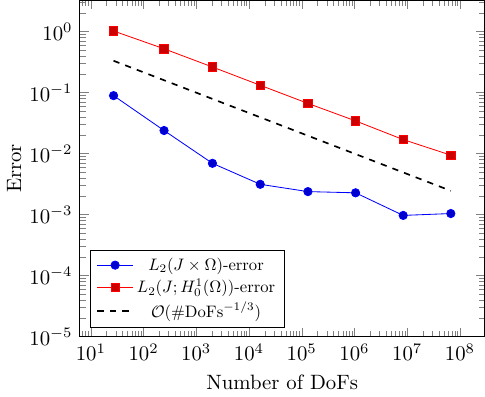}
    \end{subfigure}
\caption{The case $d=2$. Left: $L_2(\Omega)$-error for different values of $t$ versus number of DoFs in $\mathcal{X}^\delta$. Right: $L_2(J\times \Omega)$ and $L_2(J;H^1(\Omega))$-errors versus number of DoFs in $\mathcal{X}^\delta$.}\label{fig:example 1}
\end{figure}
\subsubsection{The case $d=3$}
We now consider the case $d=3$. The numerical results are shown in Figure \ref{fig:example 1 d=3}. The approximations of $u(t)$ for $t$ close to $T=1$ converge with rate $\mathcal{O}({\rm DoFs}^{-\frac12})$ in the $L_2(\Omega)$-norm, and this convergence gradually weakens as $t$ approaches $0$. As in the two-dimensional case, the error in the $L_2(J; H^1(\Omega))$-norm converges at the optimal rate $\mathcal{O}({\rm DoFs}^{-\frac14})$. 

In the remainder, only the case $d=2$ is considered.

\begin{figure}[h!]
    \begin{subfigure}{.5\textwidth}
    \centering
    \includegraphics[width=\linewidth]{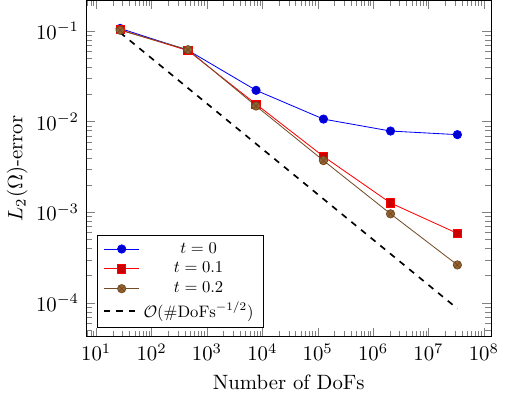}
    \end{subfigure}\hspace*{0cm}%
    \begin{subfigure}{.5\textwidth}
    \centering
    \includegraphics[width=\linewidth]{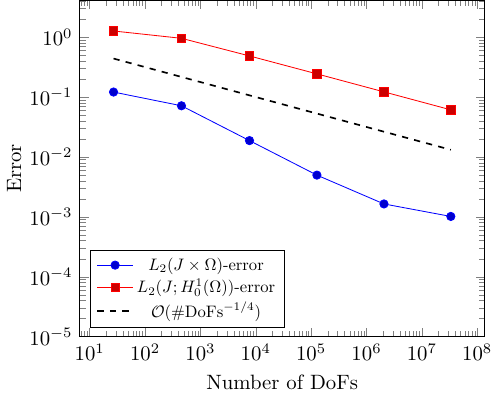}
    \end{subfigure}
\caption{The case $d=3$. Left: $L_2(\Omega)$-error for different values of $t$ versus number of DoFs in $\mathcal{X}^\delta$. Right: $L_2(J\times \Omega)$ and $L_2(J;H^1(\Omega))$-errors versus number of DoFs in $\mathcal{X}^\delta$.}\label{fig:example 1 d=3}
\end{figure}

\subsubsection{Dependence of approximations on the length of $J$}

We now give an example, which is seemingly a counterexample. However, careful study of the error analysis in Section ~\ref{sec:error analysis} provides enough insight to explain the numerical results obtained here.

We set $u(t, x, y):=\exp(2\pi^2(1-t))\sin(\pi x)\sin(\pi y)$, which satisfies $\partial_t - \triangle_x u = 0$. Furthermore, we define the interval $J_L=(1 - L, 1)$, where we set $L=\frac{1}{8}$ or $L=1$. For both values of $L$, the grid $J_L^\delta$ is constructed by dividing the interval $(0,1)$ into $2^{k+3}$ equal subintervals, and then restricting to $J_L$. The mesh $\tria^\delta$ is again constructed by $2k$ uniform bisections of the initial triangulation $\tria_0$. We choose $\eps=2^{-(k+3)}$ as our regularization parameter.

Our goal is to recover $u(t)$ at time $t = \frac{15}{16}$. In Figure~\ref{fig:length of J}, we show the $L_2(\Omega)$-errors at this time point. Although a better convergence rate is expected for $L = 1$, the errors for $L = \frac{1}{8}$ are clearly smaller. In fact, for $L=1$ we do not observe convergence at all.

This apparent counterexample can be partially explained by comparing two quantities that appear in the error estimates of Corollary~\ref{cor:errors}: the $L_2(\Omega)$-norm of the initial data and the best approximation error $\mathcal{E}_{\mathrm{appr}}(\delta)$ of $u$ in the $\mathcal{X}$-norm. Specifically, we note that
$\|\gamma_0 u\|_{L_2(\Omega)}=\exp(2\pi^2 \cdot \frac{7}{8}) \|\gamma_{\frac{7}{8}} u\|_{L_2(\Omega)},$
and that the best approximation error in the $\mathcal{X}$-norm is also significantly smaller for $L = \frac{1}{8}$ than for $L = 1$. This difference accounts for the unexpectedly favorable performance of approximating $u$ at $t=\frac{15}{16}$ with $L=\frac{1}{8}$. For really fine meshes, the method with $L=1$ is likely to perform better again.

\begin{figure}[h!]
    \centering
    \includegraphics[width=0.5\linewidth]{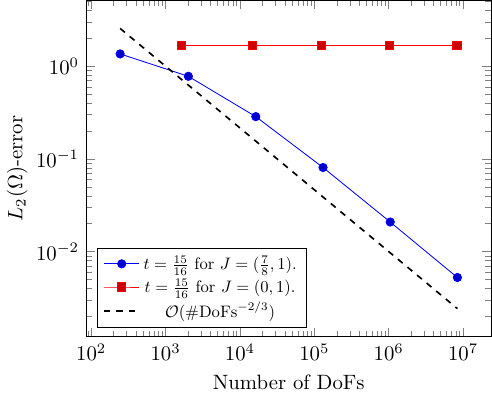}
\caption{Error of numerical approximation at $t =\frac{15}{16}$ versus number of DoFs in $\mathcal{X}^\delta$ for different intervals $J$. Exact solution is $u(t,x,y) = \exp(2\pi^2(1-t))\sin(\pi x)\sin(\pi y)$.}\label{fig:length of J}
\end{figure}

\subsection{Perturbed data}
We now demonstrate the effect of adding perturbations to the end-time data $g$. We will distinguish between two regularization strategies. The first strategy is to choose $\eps$ according to \eqref{eq:epsilon_choice}, that is, $\eps=\|g_{\rm pert}\|_{L_2(\Omega)} + {\rm DoFs}^{-\frac1d}$. The second strategy is to choose $\eps = {\rm DoFs}^{-\frac1d}$. We will first investigate the effect of adding random perturbations, and then we will add more severe perturbations. In both cases, we will consider the time interval $J=(0,\frac18)$, that is, $T=\frac18$.
\subsubsection{Random perturbation}
To achieve random perturbations, we picked a random function $g_{\rm pert}$ in $\mathcal{S}_{J^\delta}^{0,1}$ with $\|g_{\rm pert}\|_{L_2(\Omega)} = 0.01$ and replaced $g$ with $g+ g_{\rm pert}$. In Figure \ref{fig:random pert}, we compare both regularization strategies introduced above. The results indicate that for random perturbations of this type, extra regularization is not required.

\subsubsection{Difficult perturbation}
We now give an example in which it helps to choose the regularization parameter $\eps$ appropriately. Let $$u_n(t,x,y):= e^{2(n\pi)^2 (1-t)}\sin(n\pi x)\sin(n\pi y),$$ then $\|u_n(1,x,y)\|_{L_2(\Omega)} = \frac12$, while $Bu_n=0$ and $\|u_n(0,x,y)\|_{L_2(\Omega}=\frac12 e^{2 (n\pi)^2}$.  

Replacing the exact end-time data $g$ by $g + 0.05 u_n(T, \cdot, \cdot)$, we again compare both regularization strategies. Figure \ref{fig:example perturbation bad} shows that for $n=1$, additional regularization is needed to temper the data error. For larger $n$, the data error has not yet been resolved on the meshes we consider.

\begin{figure}[h!]
    \centering
    \includegraphics[width=0.5\linewidth]{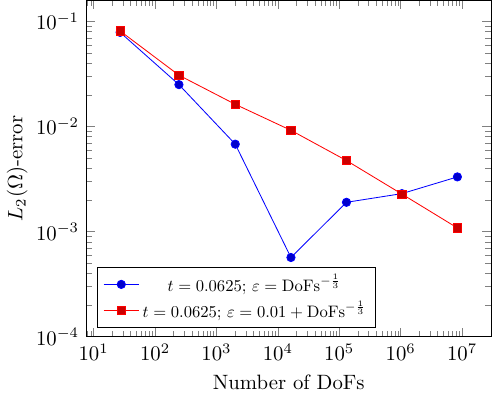}
\caption{Error of numerical approximation at $t =\frac{1}{16}$ versus number of DoFs in $\mathcal{X}^\delta$, in case of random perturbation of $g$.}\label{fig:random pert}
\end{figure}

\begin{figure}[h!]
    \centering
    \includegraphics[width=0.5\linewidth]{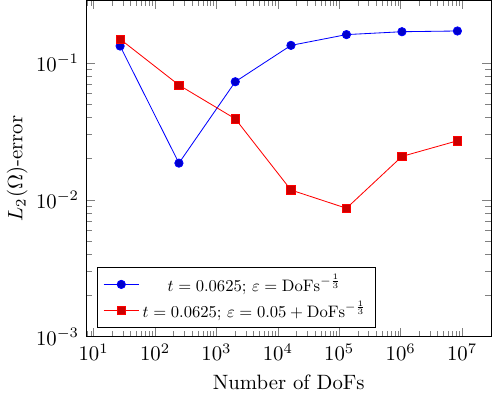}
\caption{Error of numerical approximation at $t =\frac{1}{16}$ versus number of DoFs in $\mathcal{X}^\delta$, in case of perturbation of $g$ with $0.05\cdot 2 \sin(\pi x)(\sin \pi y)$.}\label{fig:example perturbation bad}
\end{figure}

\section{Conclusion}

We have proposed and analyzed a least squares finite element method for backward parabolic problems based on a space-time variational formulation in the natural energy space. Within this framework, we established a new conditional stability estimate in the space-time $L_2(J;H^\beta(\Omega))$-norm, with $\beta\in [0,1)$. This was achieved by designing density arguments that are suitable for functions satisfying the homogeneous parabolic problem.

For the proposed numerical method, we derived \emph{a priori} error bounds expressed in terms of the data error and the best approximation error. In view of the available conditional stability estimates, these error bounds are qualitatively optimal.

To make the method fully implementable, we replaced the dual norm in the least squares functional with a computable discretized dual norm, while preserving the error bounds. The resulting discrete system can be solved efficiently using the PCG method.

\subsection*{Acknowledgment}
The author wishes to thank his advisor Rob Stevenson for
the many helpful comments.


\appendix
\section{Proof of conditional stability}\label{sec:appendix}
\begin{lemma}[\cite{MR155203}]\label{lemma:appendix}
    Let $A(t)\equiv A\in \mathcal{L}(V,V')$ be self-adjoint. Let $u\in \mathcal{X}$ satisfy $Bu=0$. Then we have
    \begin{align}\label{eq:appendix:cond stab}
        \|u(t)\|_H\leq \|u(0)\|_H^{\frac{T-t}{T}}\|u(T)\|_H^{\frac{t}{T}}.
    \end{align} 
\end{lemma}
\begin{proof}
    Define $F(t):=\|u(t)\|_H^2$. We will prove that $\log F(t)$ is a convex function. For this we compute the first and second derivative of $F$. The first derivative is given by
    \begin{align*}
        F'(t)=2\langle u(t), \partial_t u(t)\rangle_H.
    \end{align*}
    The second derivative is given by
    \begin{align*}
        F''(t)&= 2\langle \partial_t u, \partial_t u\rangle_H + 2\langle u(t), \partial_t^2 u(t)\rangle_H\\
        &=2\|\partial_t u(t)\|_H^2 -2\langle u(t), \partial_tAu(t)\rangle_H\\
        &= 2\|\partial_t u(t)\|_H^2 - 2 \langle A u(t), \partial_t u(t)\rangle_H\\
        &=4 \|\partial_t u(t)\|_H^2,
    \end{align*}
    which is finite thanks to $\partial_t u\in C((0, T];H)$ (see, Lemma~\ref{lemma:regularity positive t}).
    Now we compute 
    \begin{align*}
        F(t)F''(t)-(F'(t))^2 = 4\|u(t)\|_H^2\|\partial_t u(t)\|_H^2-4\langle u(t), \partial_t u(t)\rangle_H\geq 0,
    \end{align*}
    which verifies that $\log F(t)$ is a convex function. It follows that \eqref{eq:appendix:cond stab} holds. 
\end{proof}

\end{document}